\documentclass[12pt]{amsart}

\usepackage[colorlinks]{hyperref}

\newcommand{\IN}{\mathbb N}
\newcommand{\w}{\omega}

\newtheorem{theorem}{Theorem}
\newtheorem{corollary}{Corollary}
\newtheorem{problem}{Problem}
\newtheorem{lemma}{Lemma}
\newtheorem{claim}{Claim}

\title{Means on scattered compacta}
\author{T.~Banakh, R.~Bonnet, W.~Kubi\'s}
\thanks{
Research of the second author was supported by the European Science Foundation, ``New Frontiers of Infinity'', Short Visit Grant No. 4696. Research of the third author was supported by the GA\v{C}R grant P
201/12/0290 (Czech Republic)
}
\address{T.Banakh: Jan Kochanowski University in Kielce, Poland \textit{and} Ivan Franko National University of Lviv, Ukraine}
\email{t.o.banakh@gmail.com}
\address{R.Bonnet:  Laboratoire de Math\'ematiques,
Universit\'e de Savoie, Le Bourget-du-Lac, France}
\email{bonnet@univ-savoie.fr}
\address{W.Kubi\'s: Institute of Mathematics,
Academy of Sciences of the Czech Republic}
\email{kubisw@gmail.com}

\begin{document}
\begin{abstract}We prove that a separable Hausdorff topological space $X$ containing a cocountable subset homeomorphic to $[0,\omega_1]$ admits no separately continuous mean operation and no diagonally continuous $n$-mean for $n\ge 2$.
\end{abstract}

\maketitle

In this paper we construct a scattered compact space admitting no continuous mean operation, thus answering Problem 5 of \cite{BGR}.
By a {\em mean operation} on a set $X$ we understand any binary operation $\mu:X\times X\to X$ such that $\mu(x,x)=x$ and $\mu(x,y)=\mu(y,x)$ for all  $x,y\in X$. If, in addition, the mean operation is associative, then it is called a {\em semilattice operation}.

The mean operation is a partial case of an $n$-mean operation. A function $\mu:X^n\to X$ defined on the $n$th power of a space $X$ is called an {\em $n$-mean operation} (or briefly an {\em $n$-mean}) if
\begin{enumerate}
\item $\mu(x,\dots,x)=x$ for every $x\in X$ and
\item $\mu$ is {\em $S_n$-invariant} in the sense that $\mu(x_{\sigma(1)},\dots,x_{\sigma(n)})=\mu(x_1,\dots,x_n)$ for any permutation $\sigma$ of the set $\{1,\dots,n\}$ and any vector $(x_1,\dots,x_n)\in X^n$.
\end{enumerate}
It is clear that a mean is the same as a $2$-mean.

The problem of detecting topological spaces with (or without) a continuous mean is classical  in Algebraic Topology, see \cite{Au1}, \cite{Au2}, \cite{AC}, \cite{Eck}, \cite{H}, \cite{TT}. It particular,  due to Aumann \cite{Au1}, we know that for every $n\ge1$ the $n$-dimensional sphere admits no continuous mean. On the other hand, the 0-dimension sphere $S^0=\{-1,1\}$ trivially possesses such a mean. More generally, each zero-dimensional metrizable separable space, being homeomorphic to a subspace of the real line, admits a continuous semilattice operation.

On the other hand, there are  non-metrizable scattered compact Hausdorff spaces admitting no separately continuous semilattice operation. The simplest example is the compactification $\gamma\IN$ of the discrete space $\IN$ of natural numbers whose remainder $\gamma\IN\setminus\IN$ is homeomorphic to the ordinal segment $[0,\omega_1]$. The existence of such a compactification $\gamma\IN$ follows from the famous Parovichenko theorem \cite{Par} (saying that any compact space of weight $\le\aleph_1$ is a continuous image of $\beta\IN\setminus\IN$).

Another way to construct $\gamma\IN$ is as follows.
Consider a family $\mathcal A = (A_\alpha)_{ \alpha < \omega_1 }$
of infinite subsets of $\IN$ such that $A_\alpha\subset^* A_\beta$ for any ordinals $\alpha<\beta$. The almost inclusion $A_\alpha\subset^* A_\beta$ means that $A_\alpha\setminus A_\beta$ is finite.
Now, consider the subalgebra $B$ of $\mathcal P(\IN)$
generated  by $\mathcal A\cup \big\{ \{n\} \big\}_{ n \in \IN }$.
Then $\gamma\IN$ is the space of ultrafilters on $B$.

A bit stronger notion than the separate continuity is the diagonal continuity. A function $f:X^n\to Y$ is called {\em diagonally continuous} if for any map $g=(g_i)_{i=1}^n:X\to X^n$ whose components $g_i:X\to X$, $1\le i\le n$, are constant or identity functions the composition $f\circ g:X\to Y$ is continuous. It is clear that for a function $f:X^n\to Y$ we get the implications:
\smallskip

\centerline{continuous $\Rightarrow$ diagonally continuous $\Rightarrow$ separately continuous.}
\smallskip

A subset $A$ of a set $X$ is called {\em cocountable} if its complement $X\setminus A$ is at most countable. The following theorem is the main result of this paper.

\begin{theorem}\label{t1}If a separable Hausdorff topological space $X$ contains a cocountable subset homeomorphic to $[0,\omega_1]$, then for every $n\ge 2$ the space $X$ admits no diagonally continuous $n$-mean $\mu:X^n\to X$.
\end{theorem}

\begin{proof} This theorem will be proved by induction on $n\ge 2$. More precisely, by induction we shall prove that $X$ admits no diagonally continuous $n$-amean. A function $\mu:X^n\to X$ will be called an {\em almost $n$-mean operation} (briefly, an {\em $n$-amean}) if $\mu$ is $S_n$-invariant and the set $\{x\in X:x\ne\mu(x,\dots,x)\}$ is at most countable.

Since the space $X$ is separable, we can assume that $[\omega,\omega_1] \subset X$
has countable dense complement $D = X \setminus [\omega,\omega_1]$ that we denote by $\omega$.
So $X = \omega \cup [\omega,\omega_1) \cup \{\omega_1\}$.

The following lemma will allow us to start the inductive proof of the theorem.

\begin{lemma}\label{l1} The space $X$ admits no separately continuous $2$-amean.
\end{lemma}

\begin{proof} Assume that $\mu:X^2\to X$ is a separately continuous
$2$-amean on $X$.

Given two points $a,b\in X$ consider the closed subsets $$b/a=\{x\in [\w,\omega_1):b=\mu(a,x)\}
\mbox{ \ and \ } {\downarrow}b=\{x\in[\w,\omega_1): \mu(b,x)=x\}$$ of $[\w,\omega_1)\subset X$.
Let $$A=\{(a,b)\in \w^2:|b/a|=\aleph_1\}\mbox{ and }B=\{b\in \w:|{\downarrow}b|=\aleph_1\}.$$ Find an ordinal $\alpha_0\in[\w,\omega_1)\subset X$ such that
\begin{itemize}
\item $\mu(\alpha,\alpha)=\alpha$ for all $\alpha\ge\alpha_0$;
\item $b/a\subset[\w,\alpha_0)$ for all $(a,b)\in\w^2\setminus A$;
\item ${\downarrow}b\subset[\w,\alpha_0)$ for all $b\in \w\setminus B$.
\end{itemize}
If the set $B$ has countable closure $\bar B$ in $X$, then we will additionally assume that $\bar B\cap[\w,\omega_1)\subset[\w,\alpha_0)$.

Consider the closed unbounded subset
$$C=[\alpha_0,\omega_1)\cap\bigcap_{(a,b)\in A}b/a\;\cap\bigcap_{b\in B}{\downarrow}b$$ in $[\w,\omega_1)$ and also the open subset
$$W=\{x\in (\alpha_0,\omega_1):\exists c\in C\;\;\mu(c,x)\ne x\}$$ of $[0,\w_1)$. Observe that $W\supset C\setminus C_0$ where
$C_0=\{c\in C:\forall x\in C\; \mu(x,c)=c\}$ is a subset of $C$ containing at most one point. So, $W$ is uncountable.

Let $W_0$ stand for the dense open subset of $W$ consisting of isolated points of $W$.

\begin{claim}\label{cl1a} Any point $\alpha\in W_0$ has a neighborhood $V_\alpha\subset X$ such that $\mu\big(\{\alpha\}\times V_\alpha\big)=\{\alpha\}$.
\end{claim}

\begin{proof} Using the definition of $W$, find $c\in C$ with $\mu(c,\alpha)\ne\alpha$. Choose disjoint neighborhoods $U_{\mu(c,\alpha)},U_\alpha\subset X$ of the points $\mu(c,\alpha)$ and $\alpha$.
Replacing $U_\alpha$ by a smaller neighborhood we can assume that $\mu(\{c\}\times U_\alpha)\subset U_{\mu(c,\alpha)}$ and $U_\alpha\cap[\w,\omega_1]=\{\alpha\}$. Finally, by the separate continuity of the operation $\mu$, find a neighborhood $V_\alpha\subset U_\alpha$ such that $\mu(\{\alpha\}\times V_\alpha)\subset U_\alpha$. We claim that $\mu(\alpha,a)=\alpha$ for all $a\in V_\alpha$. This is clear if $a=\alpha$. If $a\ne\alpha$, then $a\in \w$ because $V_\alpha\cap[\w,\omega_1]=\{\alpha\}$. If $b=\mu(a,\alpha)\in \w$, then $\alpha_0<\alpha\in b/a$ and consequently, $(a,b)\in A$. It follows from $c\in C$ and $(a,b)\in A$ that $c\in b/a$, which means that $\mu(a,c)=b$. The latter equality cannot hold because $\mu(c,a)\in \mu(\{c\}\times V_\alpha)\in U_{\mu(c,\alpha)}$ while $b=\mu(\alpha,a)\in\mu(\{\alpha\}\times V_\alpha)\subset U_\alpha$. This contradiction shows that $b=\mu(a,\alpha)\in[\w,\omega_1]\cap U_\alpha=\{\alpha\}$ and hence $\mu(\alpha,a)=\mu(a,\alpha)=\alpha$.
\end{proof}

\begin{claim}\label{cl2a} The set $B$ has uncountable closure $\bar B$ in $X$.
\end{claim}

\begin{proof}
Assuming that $\bar B$ is countable, we get $\bar B\cap[\w,\omega_1)\subset[\w,\alpha_0)$ by the choice of $\alpha_0$. By Claim~\ref{cl1a}, each ordinal $\alpha\in W_0$ has a neighborhood $V_\alpha\subset X$ such that $\mu(\{\alpha\}\times V_\alpha)=\{\alpha\}$. Since $\alpha\notin\bar B$ and the set $\w$ is dense in $X$, we can pick a point $v_\alpha\in \w\cap V_\alpha\setminus \bar B$. By the Dirichlet Principle, for some point $v\in \w$ the set $W_v=\{\alpha\in W_0:v_\alpha=v\}$ is uncountable. It follows that $\mu(\alpha,v)=\mu(\alpha,v_\alpha)=\alpha$ for every $\alpha\in W_v$. Consequently, $v\in B$, which  contradicts the choice of $v=v_\alpha\notin \bar B$ for $\alpha\in W_v$.
\end{proof}

Observe that for any $c\in C$ and any $b\in B$ we get $\mu(c,b)=c$. By the separate continuity of the amean $\mu$, we get $\mu(c,b)=c$ for all $b\in\bar B$. Since $C$ and $\bar B\cap[\w,\omega_1)$ are closed uncountable subsets of $[0,\omega_1)$ the intersection $C\cap \bar B$ is uncountable and thus we can chose two distinct points $x,y\in C\cap\bar B$, for which we get $x=\mu(x,y)=\mu(y,x)=y$, which is a desired contradiction completing the proof of Lemma~\ref{l1}.
\end{proof}

The inductive step of the inductive proof of Theorem~\ref{t1} is fulfilled in the following lemma.

\begin{lemma}\label{l2} If for some $n\ge 2$ the space $X$ admits no diagonally continuous $n$-amean, then it admits no diagonally continuous $(n+1)$-amean.
\end{lemma}

\begin{proof} To derive a contradiction, assume that $X$ admits a diagonally continuous $(n+1)$-amean $\mu:X^{n+1}\to X$.

For points $\vec a\in X^{n}$ and $b\in X$ consider the closed subsets $$b/\vec a=\{x\in [\w,\w_1):b=\mu(\vec a,x)\}\mbox{ \ and \ }
{\downarrow}\vec a=\{x\in[\w,\w_1):\mu(\vec a,x)=x\}$$
of $[\w,\w_1)$. Let $$A=\{(\vec a,b)\in \w^{n}\times \w:|b/\vec a|=\aleph_1\}\mbox{ \ and \ }B=\{\vec b\in \w^{n}:|{\downarrow}\vec b|=\aleph_1\}.$$ Find a countable ordinal $\alpha_0\in[\w,\w_1)$ such that
\begin{itemize}
\item $\mu(\alpha,\dots,\alpha)=\alpha$ for every $\alpha\in[\alpha_0,\w_1)$;
\item $b/\vec a\subset [\w,\alpha_0)$ for every $(\vec a,b)\in(\w^{n}\times\w)\setminus A$, and
\item ${\downarrow}\vec b\subset [\w,\alpha_0)$ for every $\vec b\in \w^{n}\setminus B$.
\end{itemize}
It follows that
$$C=[\alpha_0,\w_1)\cap\bigg(\bigcap_{(\vec a,b)\in A}b/\vec a\bigg)\cap\bigg(\bigcap_{\vec b\in B}{\downarrow}\vec b\bigg)$$is a closed unbounded subset of $[\w,\w_1)$.

Since the space $X$ admits no diagonally continuous $n$-amean, the set $$W=\{\alpha\in[\alpha_0,\w_1):\mu(\alpha,\dots,\alpha,\w_1)\ne \alpha\}$$ is uncountable (in the opposite case the function $\nu:X^n\to X$, $\nu:(x_1,\dots,x_n)\mapsto\mu(x_1,\dots,x_n,\w_1)$, is a diagonally continuous $n$-amean on $X$, which does not exist according to our assumption).

The diagonal continuity of the function $\mu$ guarantees that the set $W$ is open in $[\alpha_0,\w_1)$. Consequently, the set $W_0$ of all isolated points of $W$ is uncountable too.

\begin{claim}\label{cl1} Each point $\alpha\in W_0$ has a neighborhood $V_\alpha\subset X$ such that for any point $x\in\w\cap V_\alpha$ there is a neighborhood $V'_{\alpha}\subset X$ of $\alpha$ such that $\mu\big(\{x\}^{n-1}\times V_\alpha'\times\{\alpha\}\big)=\{\alpha\}$.
\end{claim}

\begin{proof} By the definition of $W\supset W_0\ni \alpha$, the point $z=\mu(\alpha,\dots,\alpha,\w_1)$ differs from $\alpha$, which allows us to choose disjoint open neighborhoods $U_z$ and $U_\alpha$ of the points $z$ and $\alpha$ in $X$, respectively. Since $\alpha$ is an isolated point of $[\w,\w_1]$, we can additionally assume that $U_\alpha\cap[\w,\w_1]\subset \{\alpha\}$. It follows from $\alpha\ge\alpha_0$ that $\mu(\alpha,\dots,\alpha)=\alpha$.
The diagonal continuity of the operation $\mu$ yields a neighborhood $V_\alpha\subset X$ of $\alpha$ such that for any $x\in V_\alpha$ we get  $\mu(x,\dots,x,\alpha,\alpha)\in U_\alpha$ and  $\mu(x,\dots,x,\alpha,\w_1)\in U_z$. For every $x\in \w \cap V_\alpha$ the separate continuity of $\mu$ yields a neighborhood $V_\alpha'\subset X$ of $\alpha$ such that  $\mu(\{x\}^{n-1}\times V_\alpha'\times\{\alpha\})\in U_\alpha$ and $\mu(\{x\}^{n-1}\times V_\alpha'\times\{\w_1\})\in U_z$. Choose any $y\in V_\alpha'\cap \w$. We claim that the point $u=\mu(x,\dots,x,y,\alpha)\in U_\alpha$ belongs to $[\w,\w_1]$. Assuming the converse, we conclude that $((x,\dots,x,y),u)\in A$ and hence $\mu(x,\dots,x,y,c)=u$ for all $c\in C$. On the other hand, the separate continuity of $\mu$ and the inclusion $\mu(x,\dots,x,y,\w_1)\in U_z$ yields a point $c\in C$ with $\mu(x,\dots,x,y,c)\in U_z$. Then $u=\mu(x,\dots,x,y,c)\in U_z\cap U_\alpha=\emptyset$, which is a desired contradiction showing that $\mu(x,\dots,x,y,\alpha)=u\in [\w,\w_1]\cap U_\alpha=\{\alpha\}$.
\end{proof}

\begin{claim}\label{cl2} There is a point $x\in \w$ such that the set $$B(x)=\{y\in [\w,\w_1):\forall c\in C \; \; \mu(x,\dots,x,y,c)=c\}$$ is uncountable.
\end{claim}

\begin{proof} Assume conversely that for every $x\in\w$ the set $B(x)$ is at most countable. Then we can find an ordinal $\beta\in[\alpha_0,\w_1)$ such that $[\beta,\w_1)\cap\bigcup_{x\in\w}B(x)=\emptyset$. By Claim~\ref{cl1},  every ordinal $\alpha\in W_0\cap[\beta,\w_1)$ has a neighborhood $V_\alpha\subset X$ such that for each point $v\in \w \cap V_\alpha$ there is a neighborhood $V_\alpha'\subset X$ of $\alpha$ such that $\mu(\{v\}^{n-1}\times V_\alpha'\times\{\alpha\})=\{\alpha\}$. For every ordinal $\alpha\in W_0\cap [\beta,\w_1)$ choose a point $v_\alpha\in \w \cap V_\alpha$.
By the Dirichlet Principle, for some point $v\in\w$ the set $W_v=\{\alpha\in W_0\cap[\beta,\w_1):v_\alpha=v\}$ is uncountable. So, we can choose an ordinal $\alpha\in W_v\setminus B(v)$. For the ordinal $\alpha$ and the point $v=v_\alpha\in V_\alpha$ there is a neighborhood $V_\alpha'\subset X$ of $\alpha$ such that  $\mu(\{v\}^{n-1}\times V_\alpha'\times\{\alpha\})=\{\alpha\}$.

Since the set $B(v)$ is closed (by the separate continuity of $\mu$) and does not contain $\alpha$, we can choose a point $y\in \w\cap V'_\alpha\setminus B(v)$. For this point $y$ we get $\mu(v,\dots,v,y,\alpha)=\alpha\ge\alpha_0$, which implies $(v,\dots,v,y)\in B$ and $\mu(v,\dots,v,y,c)=c$ for all $c\in C$. The latter means that $y\in B(v)$, which contradicts the choice of $y$.
\end{proof}

By Claim~\ref{cl2}, for some $x\in \w$ the closed set $B(x)$ is uncountable. Then $C\cap B(x)$ is a closed unbounded set in $[\w,\w_1)$, which allows us to find two distinct points $y,c\in C\cap B(x)$. For these points by the $S_{n+1}$-invariance of $\mu$ we get $$c=\mu(v,\dots,v,y,c)=\mu(v,\dots,v,c,y)=y,$$ which is a desired contradiction, completing the proof of Lemma~\ref{l2}.
\end{proof}

By induction, Lemmas~\ref{l1} and \ref{l2} imply that for every $n\ge 2$ the space $X$ admits no diagonally continuous $n$-amean and hence no
diagonally continuous $n$-mean.
\end{proof}

Since each separately continuous mean $\mu:X^2\to X$ is diagonally continuous, (the proof of) Lemma~\ref{l1} implies the following corollary answering Problem~5 in \cite{BGR}.

\begin{corollary}If a separable Hausdorff topological space $X$ contains a cocountable subset homeomorphic to $[0,\omega_1)$, then $X$ admits no separately continuous mean $\mu:X^2 \to X$.
\end{corollary}

\begin{problem} Let $X$ be a separable Hausdorff topological space $X$ containing a cocountable subset homeomorphic to $[0,\omega_1)$. Does $X$ admit a separately continuous $n$-mean $\mu:X^n\to X$ for some $n\ge 3$?
\end{problem}

By the {\em $n$-th symmetric power} $SP^n(X)$ of a topological space $X$ we understand the quotient space of $X^n$ by the equivalence relation $\sim$: $(x_1,\dots,x_n)\sim(y_1,\dots,y_n)$ if there is a permutation $\sigma$ of $\{1,\dots,n\}$ such that $(y_1,\dots,y_n)=(x_{\sigma(1)},\dots,x_{\sigma(n)})$. The space $X$ is identified with the subspace $\big\{\{(x,\dots,x)\}:x\in X\big\}$ of $SP^n(X)$.

Observe that $X$ is a retract of its $n$th symmetric power $SP^n(X)$ if and only if $X$ admits a continuous $n$-mean. This observation combined with Theorem~\ref{t1} implies:

\begin{corollary} If a separable Hausdorff topological space $X$ contains a cocountable subset homeomorphic to $[0,\omega_1]$, then for every $n\ge 2$ the space $X$ is not a retract of its $n$-th symmetric power $SP^n(X)$.
\end{corollary}

The $n$-th symmetric power $SP^n(X)$ is a partial case of the $n$-th $G$-symmetric power $SP^n_G(X)$ where $G$ is a subgroup of the symmetric group $S_n$. The space $SP^n_G(X)$ is
the quotient space of $X^n$ by the equivalence relation $\sim_G$: $(x_1,\dots,x_n)\sim_G(y_1,\dots,y_n)$ if there is a permutation $\sigma\in G$ of $\{1,\dots,n\}$ such that $(y_1,\dots,y_n)=(x_{\sigma(1)},\dots,x_{\sigma(n)})$. The space $X$ is identified with the subspace $\big\{\{(x,\dots,x)\}:x\in X\big\}$ of $SP^n_G(X)$.

\begin{problem}\label{pr2} Let $X$ be a separable compact space containing a cocountable subset homeomorphic to $[0,\omega_1]$. Is $X$ a retract of $SP^n_G(X)$ for some $n\ge 2$ and some non-trivial subgroup $G\subset S_n$?
\end{problem}

Let us recall that a topological space $X$ is called {\em scattered\/} if each subspace $A\subset X$ has an isolated point.

\begin{problem} Assume that a scattered compact space $X$ admits a continuous $n$-mean for some $n\ge 2$. Does $X$ admit a continuous $n$-mean for every $n\ge2$?
\end{problem}

If $\vee:X\times X\to X$ is a semilattice operation on a set $X$, then for every $n\ge 2$ the map $\mu:X^n\to X$, $\mu(x_1,\dots,x_n)=x_1\vee\dots\vee x_n$ is an $n$-mean on $X$. So, a topological space admitting a continuous semilattice operation admits continuous $n$-means for all $n\ge 2$.

\begin{problem} Assume that a scattered compact space $X$ admits a continuous $n$-mean for every $n\ge2$. Does $X$ admit a continuous semilattice operation?
\end{problem}

It is known that each separately continuous semilattice operation on a zero-dimensional compact space is jointly continuous, see \cite[II.1.5]{HMS}.

\begin{problem} Assume a scattered compact space $X$ admits a separately continuous $n$-mean. Does $X$ admit a continuous $n$-mean?
\end{problem}

The constructions of symmetric powers $SP^n$ and $SP^n_G$ are examples of normal functors in the category $\mathbf{Comp}$ of compact Hausdorff spaces and their continuous maps, see \cite[2.3.2]{TZ}.
So, the following problem can be considered as a general version of Problem~\ref{pr2}. 

\begin{problem} Is a normal functor $F:\mathbf{Comp}\to\mathbf{Comp}$ a power functor if each (scattered) compact space $X$ is a retract of $F(X)$?
\end{problem}

According to \cite{BGR} and \cite{BK}, another example of a scattered compact space admitting no separately continuous semilattice operation is the  Mr\'owka space $\psi\IN$. By definition, the {\em Mr\'owka space} is the Stone space of the Boolean algebra generated by $\mathcal A\cup\big\{\{n\}\big\}_{n\in\IN}$ for some maximal almost disjoint family $\mathcal A$ of infinite subsets of $\IN$.

\begin{problem} Does the Mr\'owka space $\psi \IN$ admit a (separately) continuous $n$-mean for some $n\ge 2$?
\end{problem}

\end{document}